\documentclass[reqno]{amsart}

\usepackage{amsmath}
\usepackage[foot]{amsaddr} 
\usepackage{amsfonts}
\usepackage{pdfsync}
\usepackage{color}
\usepackage{graphicx}
\usepackage{amssymb}
\usepackage{epsf}
\usepackage{eucal}
\usepackage{amscd}
\usepackage{graphics}
\usepackage{amsthm}
\usepackage{fancyhdr}
\usepackage{delarray}
\usepackage{fullpage}
\usepackage{enumerate}
\usepackage{paralist}
\usepackage[all]{xy}
\usepackage{tikz}
\usepackage{verbatim}
\usepackage[noadjust]{cite}
\usepackage{mathrsfs}
\usepackage{bm}

\usepackage{hyperref}

\theoremstyle{plain}

\newtheorem{thm}{Theorem}[section]
\newtheorem{lem}[thm]{Lemma}
\newtheorem{cor}[thm]{Corollary}
\newtheorem{prop}[thm]{Proposition}

\theoremstyle{definition}

\newtheorem{defn}[thm]{Definition}
\newtheorem{conj}[thm]{Conjecture}

\newtheorem{quest}[thm]{Question}

\newcommand{\thisconjname}{}
\newtheorem*{genericconj*}{\thisconjname}
\newenvironment{namedconj*}[1]
{\renewcommand{\thisconjname}{#1}%
	\begin{genericconj*}}
	{\end{genericconj*}}

\newcommand{\thislemname}{}
\newtheorem*{genericlem*}{\thislemname}
\newenvironment{namedlem*}[1]
{\renewcommand{\thislemname}{#1}%
	\begin{genericlem*}}
	{\end{genericlem*}}

\theoremstyle{remark}

\newtheorem{rem}{Remark}

\newtheorem{claim}{Claim}

\newtheorem{hyp}{Hypothesis}

\newcommand{\N}{\mathbb{N}}
\newcommand{\R}{\mathbb{R}}

\newcommand{\sx}{\sigma_x}
\newcommand{\bx}{\beta_x}
\DeclareMathOperator{\Aut}{Aut}
\DeclareMathOperator{\dom}{dom}

\DeclareMathOperator{\id}{id}
\DeclareMathOperator{\DL}{DL}

\DeclareMathOperator{\Aff}{Aff}
\DeclareMathOperator{\QI}{QI}
\DeclareMathOperator{\h}{\mathfrak{h}}
\DeclareMathOperator{\gm}{gm}
\DeclareMathOperator{\im}{im}

\usepackage{color}
\definecolor{darkgreen}{cmyk}{1,0,1,.2}
\definecolor{m}{rgb}{1,0.1,1}
\definecolor{darkorchid}{rgb}{0.6, 0.2, 0.8}

\newdimen\theight
\def\TeXref#1{%
	\leavevmode\vadjust{\setbox0=\hbox{{\tt
				\quad\quad  {\small \textrm #1}}}%
		\theight=\ht0
		\advance\theight by \lineskip
		\kern -\theight \vbox to
		\theight{\rightline{\rlap{\box0}}%
			\vss}%
}}%

\begin{document}

\title[Coarse distinguishability of graphs]{Coarse distinguishability of graphs with symmetric growth}

\author[J.A. \'{A}lvarez L\'{o}pez]{Jes\'{u}s A. \'{A}lvarez L\'{o}pez}
\address{Jes\'{u}s A. \'{A}lvarez L\'{o}pez, Departamento e Instituto de Matem\'{a}ticas, Facultade de Matem\'{a}ticas, Universidade de Santiago de Compostela, 15782 Santiago de Compostela, Spain. \textit{E-mail}: \textup{\texttt{jesus.alvarez@usc.es}}}

\author[R. Barral Lij\'o]{Ram\'on Barral Lij\'o}
\address[Corresponding author]{Ram\'on Barral Lij\'o, Research Organization of Science and Technology,	Ritsumeikan University, Nojihigashi 1-1-1, Kusatsu-Shiga, 525-8577, Japan. \textit{E-mail}: \textup{\texttt{ramonbarrallijo@gmail.com}}}

\author[H. Nozawa]{Hiraku Nozawa}
\address{Hiraku Nozawa, Department of Mathematical Sciences, Colleges of Science and Engineering, Ritsumeikan University,	Kusatsu-Shiga, 525-8577, Japan. \textit{E-mail}: \textup{\texttt{hnozawa@fc.ritsumei.ac.jp}}}



\date{\today}

\begin{abstract}
Let $X$ be a connected, locally finite graph with symmetric growth. We prove that there is a vertex coloring $\phi\colon X\to\{0,1\}$ and some  $R\in\N$ such that every automorphism $f$ preserving $\phi$ is $R$-close to the identity map; this can be seen as a coarse geometric version of symmetry breaking. We also prove that the infinite motion conjecture  is true for graphs where at least one vertex stabilizer $S_x$ satisfies the following condition: for every non-identity automorphism $f\in S_x$, there is a sequence $x_n$ such that $\lim d(x_n,f(x_n))=\infty$.
\end{abstract}

\maketitle

\tableofcontents

\section{Introduction}

A (not necessarily proper) vertex coloring $\phi$ of a graph is \emph{distinguishing} if the only automorphism that preserves $\phi$ is the identity. This notion was first introduced in~\cite{Babai} under the name \emph{asymmetric coloring}, where it was proved that $2$ colors suffice to produce a distinguishing coloring of a regular tree. Later, Albertson and Collins~\cite{AlbertsonCollins} defined the  \emph{distinguishing number} $D(X)$ of a graph $X$ as the least number of colors needed to produce a distinguishing coloring. The problem of calculating  $D(X)$ and variants thereof has accumulated an extensive literature in the last 20 years, see e.g.~\cite{HuningImrichKloasSchreberTucker, Tucker, LehnerMoller, Lehner, LehnerPilsniakStawiski, AlvarezBarral} and references therein.

One of most important open problems in graph distinguishability is the Infinite Motion Conjecture of T.~Tucker. Let us introduce some preliminaries: The \emph{motion} $m(f)$ of a graph automorphism $f$ is the cardinality of the set of points that are not fixed by $f$. 
For a graph $X$ and a subset $A\subset \Aut(X)$, the motion of $A$ is $m(A)=\inf\{m(f)\mid f\in A,\ f\neq \id\}$, and the motion of $X$ is $m(X)=m(\Aut(X))$. 
A probabilistic argument yields the following result for finite graphs.

\begin{lem}[Motion Lemma,~\cite{RussellSundaram}] \label{l:motionlemma}
	If $X$ is a finite graph and $2^{m(X)}\geq |\Aut(X)|^2$, then $D(X)\leq 2$.
\end{lem}

We always have $|\Aut(X)|^2\leq 2^{\aleph_0}$ when $X$ is countable, which motivates the following generalization.

\begin{conj}[Infinite motion conjecture,~\cite{Tucker}]\label{c:infinitemotionconjecture}
	If $X$ is a connected, locally finite graph with infinite motion, then $D(X)\leq 2$.
\end{conj}
The condition of local finiteness cannot be omitted~\cite{LehnerMoller}; note also that every connected, locally finite graph is countable.  This conjecture has been confirmed for special classes of graphs: F.~Lehner proved it in~\cite{Lehner} for graphs with growth at most $\mathcal{O}(2^{(1-\epsilon)\frac{\sqrt{n}}{2}})$ for some $\epsilon>0$,\footnote{The notation $f=\mathcal{O}(g)$ is used if there are $C,N$ such that $f(x)\leq Cg(x)$ for all $x>N$.} and later, together with M.~Pil\'sniak and M.~Stawiski~\cite{LehnerPilsniakStawiski}, for graphs with degree less or equal to five.

The aim of this paper is to introduce a large-scale-geometric version of  distinguishability for colorings, and to prove the existence of such colorings in graphs whose growth functions are large-scale symmetric. This will result in a proof of Conjecture~\ref{c:infinitemotionconjecture} for graphs with a vertex stabilizer $S_x$ satisfying that, for every  automorphism $f\in S_{x}\setminus \{\id\}$, there is a sequence $x_n$ such that $d(x_n,f(x_n))\to\infty$; we can regard this condition as a geometric refinement of having infinite motion.

Let $X$ and $Y$ be connected graphs, endowed with their canonical $\N$-valued\footnote{We will use the convention that $0\in\N$.} metric. In the context of coarse geometry (see~\cite{Roe} for a nice exposition on the subject), two functions $f,g\colon X\to Y$ are \emph{$R$-close} ($R\geq0$) if $d(f(x),g(x))\leq R$ for all $x\in X$, and we say that $f$ and $g$ are \emph{close} if they are $R$-close for some $R\geq0$. Let $\QI(X)$ denote the group of closeness classes of quasi-isometries (in the sense of Gromov) $f\colon X\to X$, and let $\iota\colon\Aut(X)\to\QI(X)$ denote the natural map that sends every automorphism to its closeness class. 
We can adapt the notion of distinguishing coloring to this setting as follows:

\begin{defn}
A coloring $\phi\colon X\to \N$ is \emph{coarsely distinguishing} if every $f\in\Aut(X,\phi)$ is close to the identity; that is, $\iota(\Aut(X,\phi))=\{[\id_X]\}$.
\end{defn}

This new definition begs the following question: which connected, locally finite graphs admit a coarsely distinguishing coloring by two colors? In Section~\ref{ss:counter} we present a simple example of a graph that does not admit such a coloring. The first main result of this paper shows that  graphs with \emph{symmetric growth} (see Definition~\ref{d:uniform}) admit  coarsely distinguishing colorings by two colors; this condition is satisfied by vertex-transitive graphs and, more generally, coarsely quasi-symmetric graphs~\cite[Def.~3.16 \& Cor.~4.17]{AlvarezCandel}.  

\begin{thm}\label{t:main}
	Let $X$ be a connected, locally finite graph of symmetric growth. Then there are $R\in\N$ and $\phi\colon X\to \{0,1\}$ such that every $f\in\Aut(X,\phi)$ satisfies $d(x,f(x))\leq R$ for all $x\in X$. 
\end{thm} 

Note that we obtain a uniform closeness parameter $R$ for all $f\in\Aut(X,\phi)$; furthermore, we make no assumption on the motion of the graph. A slight modification of the proof of Theorem~\ref{t:main} proves the infinite motion conjecture for graphs $X$ containing a vertex $x\in X$ such that the restriction $\iota\colon S_x\to \QI(X)$ is injective.
Let us rephrase this condition in a language closer to the statement of Conjecture~\ref{c:infinitemotionconjecture}.
Let $X$ be a connected graph and let $f\in\Aut(X)$. The \emph{geometric motion} of $f$ is then $\gm(f)=\sup\{d(x,f(x))\mid x\in X\}$; for a subset $A\subset\Aut(X)$, the geometric motion of $A$ is $\gm(A)=\sup\{\gm(f)\mid f\in A,\ f\ne\id\}$.  The definition of the ``closeness" relation  for functions yields that the restriction $\iota\colon A\to \QI(X)$ is injective if and only if $\gm(A)=\infty$.
The second main result of the paper then reads as follows.

\begin{thm}\label{t:motion}
	Let $X$ be a connected, locally finite graph with symmetric growth. If $m(X)=\infty$ and there exists $x\in X$ such that $\gm(S_x)=\infty$, then $D(X)\le 2$.
\end{thm}

In Sections~\ref{ss:dl} and~\ref{ss:boundedcycle} we present two families of graphs satisfying the hypothesis of Theorem~\ref{t:motion}: the \emph{Diestel-Leader graphs} $\DL(p,q)$, $p,q\geq 2$, and graphs with \emph{bounded cycle length}. As far as the authors know, the existence of distinguishing colorings by $2$ colors for these graphs had not been established before. The origin of Diestel-Leader graphs goes back to the following question, posed in~\cite{SoardiWoess,Woess} by W.\ Woess:
\begin{quest}\label{q:qi}
	Is there a locally finite vertex-transitive graph that is not quasi-isometric to the Cayley graph of some finitely generated group?
\end{quest}
R.\ Diestel and I.\ Leader introduced in~\cite{DiestelLeader} the graph $\DL(2,3)$, and conjectured that it satisfies the conditions of Question~\ref{q:qi}. A.\ Eskin, D.\ Fisher, and K.\ Whyte proved in~\cite{EskinFisherWhyteQ,EskinFisherWhyteI,EskinFisherWhyteII} that in fact all graphs $\DL(p,q)$ with $p\neq q$ answer Question~\ref{q:qi} positively. On the other hand, graphs with bounded cycle length are hyperbolic (in the sense of Gromov) and contain as examples free products of finite graphs.

We can sketch the idea behind the proof as follows: Choose a suitable $R>0$ and a subset $Y\subset X$ such that $d(x,Y)\leq R$ for all $x\in X$. Suppose that there is a partial coloring $\psi$ by two colors such that, if $\phi\colon X\to\{0,1\}$ is an extension of $\psi$ and $f$ is an automorphism of $X$ preserving $\phi$,  then $f(Y)=Y$. Thus we can regard every extension $\phi$ of $\psi$ as a coloring $\bar\phi\colon Y\to\N$ by more than two colors. The hypothesis of symmetric growth  ensures that, for $R$ large enough, we have  sufficiently many local extensions of $\psi$ around every point $y\in Y$ so that, gluing them, we can find a global extension $\phi$ with $\bar\phi$ distinguishing. Theorems~\ref{t:main} and~\ref{t:motion}  then follow from a simple geometrical argument. In general, we cannot find a partial coloring $\psi$ as above, but the same idea works with minor modifications; this technique is  similar to that used in~\cite{AlvarezBarral}.

The outline of the paper is as follows: In the next section we introduce some preliminaries to be used in the proof of the main theorems, which comprises Sections~\ref{s:coloring} and~\ref{s:growth}. Finally, Section~\ref{s:examples} contains several examples illustrating some of the concepts that appear in the paper.

\section{Preliminaries}\label{s:preliminaries}

In what follows we only consider undirected, simple graphs, so there are no loops and no multiple edges.
We  identify a graph with its vertex set, and by abuse of notation we  write $X=(X,E_X)$. The \emph{degree} of a vertex $x\in X$, $\deg x$, is the number of edges incident to $x$, and the degree of $X$ is $\deg X=\sup\{\deg x\mid x\in X\}$. A graph $X$ is \emph{locally finite} if $\deg x<\infty$ for all $x\in X$. 
A \emph{path} $\gamma$ in $X$ of \emph{length} $l\in\N$ is a finite sequence $x_0,x_1,\ldots,x_l$ of vertices such that $x_{i-1}E_{X}x_{i}$ for all $i=1,\ldots,l$; when the sequence of vertices is infinite, we call $\gamma$ a \emph{ray}. We may also think of a path (respectively, a ray) as a function $\sigma\colon \{0,\ldots,n\}\to X$ (respectively, $\sigma\colon \N\to X$). A graph is \emph{connected} if every two vertices can be joined by a path. 
All graphs in this paper are assumed to be  connected and locally finite, hence countable.  We consider every graph to be endowed with its canonical $\N$-valued metric, where $d(x,y)$ is the length of the shortest path joining $x$ and $y$; a length-minimizing path is termed a \emph{geodesic path}.

A \emph{partial coloring} of a graph $X$ is a map $\psi\colon Y\to \N$, where $Y\subset X$; if $Y=X$, we simply call $\psi$ a \emph{coloring}.  We use the term (partial) \emph{$2$-coloring} when $\psi$ takes values in $\{0,1\}$. For every graph $X$ and  coloring $\phi\colon X\to\N$, let  $\Aut(X,\phi)$ denote the group of automorphisms $f$ of $X$ satisfying $\phi=\phi\circ f$. 
A coloring $\phi\colon X\to \N$ is  \emph{distinguishing} if $\Aut(X,\phi)=\{\id\}$.

For a graph $X$, $x\in X$, and $r\in\N$, let  
\[
D(x,r)=\{\, y\in X\mid d(y,x)\leq r\,\}, \qquad S(x,r)=\{\,y\in X\mid d(y,x)=r\,\}
\]
denote the \emph{disk}  and the \emph{sphere} of center $x$ and radius $r$, respectively. 
We may write $D_X(x,r)$ for $D(x,r)$ when the ambient space $X$ is not clear from context.
A subset $Y$ of $X$ is  \emph{$R$-separated} ($R>0$)  if $d(y,y')\geq R$ for all $y,y'\in Y$ with $y\neq y'$; it is \emph{$R$-coarsely dense} if, for every $x\in X$, there is some $y\in Y$ with $d(x,y)\leq R$. The next result follows from a simple application of Zorn's Lemma. 
\begin{lem}[E.g.~{\cite[Cor.~2.2.]{AlvarezBarral}}]\label{l:delone}
	Let $X$ be a graph and let $R>0$. For every $x\in X$,  there is a $(2R+1)$-separated, $2R$-coarsely dense subset $Y\subset X$ containing $x$.
\end{lem}

Let $\bx\colon \N\to\N$ and $\sx\colon\N\to\N$ be the functions defined by
\[
\bx(r)=|D(x,r)|, \qquad \sx(r)=|S(x,r)|.
\] 
Given two non-decreasing functions $f,g\colon \N \to \R^+$, $f$ is \emph{dominated} by $g$ if there are integers $k,l,m$ such that $f(r)\leq kg(lr)$ for all $r\geq m$.  Two functions have the same \emph{growth type} if they dominate one another. The growth type of $\bx$ does not depend on the choice of point $x\in X$, so every graph has a well-defined growth type. The  functions $\bx$, $x\in X$, however, may not dominate one another with a uniform choice of constants, which motivates the following definition.

\begin{defn}[{\cite[Def.~4.13]{AlvarezCandel}}]\label{d:uniform}
	A graph $X$ has \emph{symmetric growth} if there are $k,l,m\in\N$ such that $\bx(r)\leq k\beta_y(lr)$ for all $r\geq m$ and $x,y\in X$.
\end{defn}

\begin{lem}
	If $X$ has symmetric growth, then $\deg X<\infty$.
\end{lem}
\begin{proof}
Let $x\in X$, then we have $\deg y< \beta_y(1)\leq k\beta_x(lm)<\infty$ for every $y\in X$.
\end{proof}

Let $X$ be a graph with $\Delta:=\deg X<\infty$, then the following  holds for all $x\in X$ and $r\geq 1$~\cite[Lem.~2.12]{AlvarezBarral}:
\begin{align}
\sx(1)&\leq \Delta,    \label{sx1}       \\
\sx(r+1)&\leq \sx(r)(\Delta-1), \label{sxrplus1}\\
\sx(r+1)&\leq \Delta(\Delta-1)^{r}. \label{sxr}
\end{align}
We will later fix a graph with $\Delta>2$; note that in this case $\Delta/(\Delta-2)\le 3$, so
\begin{equation}\label{bxr}
\beta_x(r) \le 1 + \Delta \sum_{s=0}^{r-1} (\Delta -1)^s 
= 1 +  \frac{\Delta((\Delta -1)^r - 1)}{\Delta - 2} 
\le 1 + 3 (\Delta - 1)^r - 1 
= 3 (\Delta - 1)^r.
\end{equation}

We say that $X$ has \emph{exponential growth} if $\liminf \frac{\log \bx(r)}{r}>0$ for some, and hence all $x\in X$, else it has \emph{subexponential growth}. The following lemmas have elementary proofs.

\begin{lem}\label{l:uniformexp}
	Let $X$ be a graph with symmetric exponential growth. Then there are $k,l,m\in\N$ such that $e^{r}\leq k\bx(lr)$ for all $x\in X$ and $r\geq m$.
\end{lem}

\begin{lem}\label{l:uniformsub}
	If $X$ has symmetric subexponential growth, then, for every $a,b>0$, there is some $m\in\N$ such that $\bx(r)\leq ae^{br}$ for all $x\in X$ and $r\geq m$.
\end{lem}

\section{Construction of the coloring}\label{s:coloring}

Let $R$ be a large enough odd number, to be determined later.  Let $Y$ be a $(2R+1)$-separated, $2R$-coarsely dense subset of $X$; we define a  graph structure $E_Y$ on $Y$ as follows:
\begin{equation}\label{ey}
yE_Yy'\quad \text{if and only if}\quad 0<d(y,y')\leq 4R+1. 
\end{equation}

\begin{lem}\label{l:yconnected}
	The graph $(Y,E_Y)$ is connected with $\deg_Y y\leq |D_X(y,4R+1)|-1$ for all $y\in Y$.
\end{lem}
\begin{proof}
	The inequality follows trivially from~\eqref{ey}, so let us prove that $Y$ is connected. Let $y,y'\in Y$, and let $(y,x_1,\ldots,x_{n-1},y')$ be a path in $X$. Since $Y$ is $2R$-coarsely dense, for every $i=1,\ldots,n$ there is some $y_i\in Y$ with $d_X(x_i,y_i)\leq 2R$. The triangle inequality and~\eqref{ey} then yield that $(y,y_1,\ldots,y_{n-1},y')$ is a path on $(Y,E_Y)$.
\end{proof}

Recall that $R$ is a large enough odd number, so assume $R\geq 5$. Let
\begin{equation}\label{defnab}
A=\{\, 2n\mid 2\leq n\leq\frac{R-1}{2}\, \},\qquad B=\{\,2n+1\mid 1\leq n\leq \frac{R-1}{2} \,\},
\end{equation}
and, for $r\leq R$, let  	
\[
D(Y,r)=\bigcup_{y\in Y}D(y,r), \qquad S(Y,r)=D(Y,r)\setminus D(Y,r-1)=\bigcup_{y\in Y}S(y,r),
\]
where the last equality holds because $Y$ is $(2R+1)$-separated. Let us define a partial coloring
\[
\psi\colon X\setminus \bigcup_{r\in B}S(Y,r)\to \{0,1\}
\]
as follows (Cf.~\cite[Lem.~3.2]{CunoImrichLehner}, see Figure~\ref{f:phi} for an illustration):
\begin{equation}\label{defnpsi}
\psi(x)=\begin{cases}
0, &x\in \bigcup_{r=0,1}S(Y,r),\\
1, &x\in S(Y,2),\\
1, &x\in \bigcup_{r\in A}S(Y,r),\\
1, &x\notin D(Y,R).
\end{cases}
\end{equation}

\begin{figure}[th]
	\includegraphics{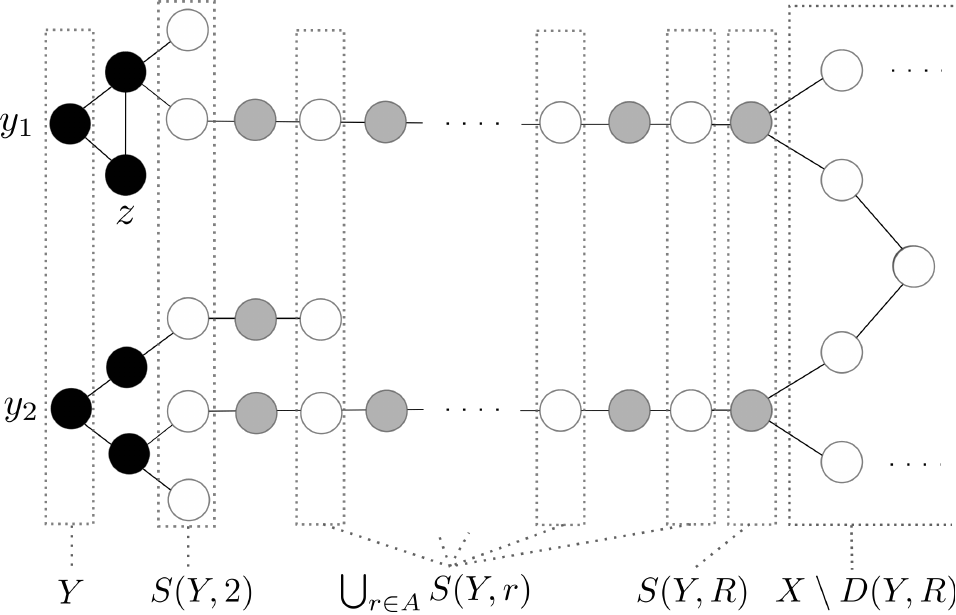}
	\caption{An illustration of the coloring $\psi$, where $y_1,y_2\in Y$, black represents the color $0$, and white represents $1$. The grey vertices are those where $\psi$ is not defined.}
	
	\label{f:phi}
\end{figure}

\begin{lem}[Cf.~{\cite[Lem.~3.2.]{CunoImrichLehner}}]\label{l:phiyprime}
	Let $\phi\colon X\to \{0,1\}$ be an extension of $\psi$, and let $f\in\Aut(X,\phi)$. For each $y\in Y$, there is some $\bar y\in Y$ such that $d(\bar y,f(y))\leq 1$ and $d(z,\bar y)=d(z,f(y))$ for all $z\in X\setminus\{\bar y,f(y)\}$.
\end{lem}
\begin{proof}
    Let 
    \[
    Y' = \{ \, z \in X \mid \phi (z') = 0\ \text{for all}\ z' \in D(z,1) \, \},
    \]
    then~\eqref{defnpsi} yields $Y' \subset D(Y,1)$, and clearly $f(Y') = Y'$ for all $f \in \Aut(X,\phi)$.  For $y\in Y$, let $\bar{y}$ be the unique vertex in $Y$ which is adjacent to $f(y)$. We have $\phi(z) = 0$ for every vertex $z\in D(f(y),1)$ and $D(f(y),1)\subset D(\bar y,2)$, so $D(f(y),1)\subset D(\bar y,1)$ by~\eqref{defnpsi}. Since $D(\bar y,1)\subset D(f(y),2)$, we also get $D(\bar y,1)\subset D(f(y),1)$, and the result follows.
\end{proof} 

\begin{cor}\label{c:motion}
	If $X$ has infinite motion, then $f(Y)=Y$.
\end{cor}
\begin{proof}
	 Let $f\in\Aut(X,\phi)$ and suppose $f(y)\neq \bar y$. By the previous lemma we have $S(f(y),1)=S(\bar y,1)$, so there is a non-trivial automorphism  exchanging $f(y)$ and $\bar y$ and leaving all other vertices in $X$ fixed. This contradicts the assumption that $X$ has infinite motion.
\end{proof}
\begin{rem}
Note that there might be automorphisms $f\in\Aut(X,\phi)$ with $f(Y)\neq Y$  when $m(X)<\infty$. The graph in Figure~\ref{f:phi} provides such an example: the map $f$ that interchanges $y_1$ and $z$ and leaves the rest of vertices fixed is an automorphism preserving $\psi$, but $f(Y)\neq Y$.
\end{rem}

Since $\dom\psi=X\setminus\bigcup_{r\in B}S(Y,r)$, an extension of $\psi$ to $X$ is the same thing as a coloring of $\bigcup_{r\in B}S(Y,r)$; for such an extension $\phi$, let $\bar \phi$ denote the induced coloring $Y\to \prod_B \N$ defined by
\begin{equation}\label{tildephi}
\bar \phi(y)=(\bar\phi_r(y))_{r\in B},\quad \text{where}\ \bar\phi_r(y)=|S(y,r)\cap \phi^{-1}(1)|.
\end{equation}

\begin{lem}\label{l:tildephiexists}
If $\xi:=(\xi_r)_{r\in B}\colon Y\to \prod_B \N$ is such that $\xi_r(y)\leq \sigma_y(r)$ for every $y\in Y$, then there is at least one extension $\phi$ satisfying $\bar \phi=\xi$.
\end{lem}
\begin{proof}
	Since $Y$ is $(2R+1)$-separated, the spheres $S(y,r)$, $y\in Y$, $r\in B$, are pairwise disjoint. Thus we can define $\phi$ independently over each sphere $S(y,r)$ by coloring $\xi_r(y)$ vertices with the color $1$ and the rest with the color $0$. 
\end{proof}

\begin{lem}\label{l:auto}
	For each extension $\phi\colon X\to\{0,1\}$ of $\psi$ and every automorphism $f\in\Aut(X,\phi)$, there is a unique automorphism $\bar f\in\Aut(Y,\bar\phi)$ such that $d(\bar f(y),f(y))\leq 1$ for all $y\in Y$.
\end{lem}
\begin{proof}
	Let $\bar f$ be defined by the formula $\bar f(y)=\bar y$, where $\bar y\in Y$ denotes the point given by Lemma~\ref{l:phiyprime}. This point satisfies $d(\bar f(y),z)=d(f(y),z)$ for all $z\in X\setminus \{f(y), \bar f(y)\}$, so 
	\[
	d(y,y')=d(f(y),f(y'))=d(\bar f(y),\bar f(y'))
	\]
	for every $y,y'\in Y$, $y\neq y'$. This equation and~\eqref{ey} yield that $\bar f$ is an automorphism of $Y$; moreover,
	\[f(S(y,r))=S(f(y),r)=S(\bar f(y),r)\]
	 for $r\geq 1$ by Lemma~\ref{l:phiyprime}, so $\bar f$ preserves $\xi$ by~\eqref{tildephi}.
\end{proof}

\begin{prop}\label{p:main}
	If $X$ has symmetric growth, then we can choose $R$ large enough so that $\prod_{r\in B}(\sx(r)+1)> \bx(4R+1)$ for all $x\in X$.
\end{prop}

In order to keep with the flow of the argument, we defer the proof of Proposition~\ref{p:main} to Section~\ref{s:growth}.
Assume for the remainder of this section that $X$ has symmetric growth and that $R$ has been chosen satisfying the statement of Proposition~\ref{p:main}.

\begin{prop}\label{p:aperiodic}
	There is a distinguishing coloring $\xi:=(\xi_r)_{r\in B}\colon Y\to \prod_B \N$ such that $\xi_r(y)\leq \sigma_y(r)+1$.
\end{prop}
\begin{proof}
	Choose a spanning tree $T$ for $(Y,E_Y)$ and a root $y_0\in Y$. In order to define $\xi$, first let $\xi(y_0)=(0,\ldots,0)$. Every $y\in Y$ with $y\neq y_0$ has at most $|D_X(y,4R+1)|-1$ siblings in $T$ by Lemma~\ref{l:yconnected}. Using Proposition~\ref{p:main}, we can define $\xi$ so that $\xi(y)\neq (0,\ldots,0)$ for all $y\neq y_0$, and every vertex is colored differently from its siblings in $T$. It can be easily checked that such a coloring is distinguishing~\cite[Lem.~4.1]{CollinsTrenk}.
\end{proof}

\begin{proof}[Proof of Theorem~{\normalfont\ref{t:main}}]
	Lemma~\ref{l:tildephiexists} and Proposition~\ref{p:aperiodic} prove the existence of some  $\phi\colon X\to\{0,1\}$ extending $\psi$ and such that $\bar\phi\colon Y\to\N$ is distinguishing. By Lemma~\ref{l:auto}, every $f\in\Aut(X,\phi)$ satisfies $d(f(y),y)\leq 1$ for all $y\in Y$. Since $Y$ is $2R$-coarsely dense, the triangle inequality yields $d(x,f(x))\leq 4R+1$ for all $x\in X$.	
\end{proof} 

\begin{proof}[Proof of Theorem~{\normalfont\ref{t:motion}}]
	Let $X$ have infinite motion and pick $x\in X$ so that $S_x$ has infinite geometric motion; Lemma~\ref{l:delone} ensures that we can choose $Y$ so that $x\in Y$.  Using Lemma~\ref{l:tildephiexists} and Proposition~\ref{p:aperiodic}, we construct a coloring $\phi\colon X\to\{0,1\}$ extending $\psi$ and such that $\bar\phi$ is distinguishing. Since $X$ has infinite motion, Corollary~\ref{c:motion} yields  $f(Y)=Y$  for every $f\in \Aut(X,\psi)$. Moreover, Lemma~\ref{l:auto} and the fact that $\bar{\phi}$ is distinguishing show that $f|_Y=\id_Y$, so  $\Aut(X,\phi)\subset S_x$. Since $\gm(S_x)=\infty$ by hypothesis,  $\gm(\Aut(X,\phi))=\infty$. But $Y$ is a $2R$-coarsely dense subset and is fixed pointwise by every automorphism $f$, so the triangle inequality yields $d(x,f(x))\leq 4R$ for all $x\in X$, a contradiction.
\end{proof}

\section{Growth estimates}\label{s:growth}

In this section we assume that $X$ is a graph with symmetric growth. The results in~\cite{Lehner} and~\cite{LehnerPilsniakStawiski} proved the existence of distinguishing $2$-colorings when either $\deg X\leq 5$, or $\beta_x=\mathcal{O}(2^{(1-\epsilon)\sqrt{n}/2})$. 
Since a distinguishing coloring is already coarsely distinguishing,  we  assume  $\Delta:=\deg X>2$ and that $X$ satisfies the following growth condition, which is weaker than having symmetric growth with $\bx(r)\neq\mathcal{O}(2^{(1-\epsilon)\sqrt{r}/2})$.

\begin{hyp}\label{h:lowerbound}
	There is an increasing sequence $(r_n)_{n\in\N}$ such that $\bx(r_n)\geq 2^{\sqrt{r_n}/4}$ for all $n\in\N$ and $x\in X$.
\end{hyp}

\begin{lem}\label{l:uniformsubexp}
	For every $\epsilon>0$, there is $R\in\N$ such that $\epsilon\bx(r)>r$ for all $r\geq R$ and $x\in X$.
\end{lem}
\begin{proof}
	Let $x\in X$, let $\epsilon>0$, and choose $n\in\N$ such that $3r_n<\epsilon^22^{\sqrt{r_n}}$.
	Let $r\geq 3r_n+1$, and let $m=\lfloor \frac{r-r_n}{2r_n+1}\rfloor$.
	Choose points $x_1,\ldots,x_m$ such that $d(x,x_i)=i(2r_n+1)$ for all $i=1,\ldots,m$, so that the disjoint union $\bigcup_{i=1}^{m}D(x_i,r_n)$ is contained in $D(x,r)$. Now Hypothesis~\ref{h:lowerbound} yields
	\[
	\epsilon^2\bx(r)\geq \epsilon^2\sum_{i=1}^{m}\beta_{x_i}(r_n)\geq \epsilon^2 m2^{\sqrt{r_n}}> m3r_n=\left\lfloor \frac{r-r_n}{2r_n+1}\right\rfloor3r_n\geq \left(\frac{r-r_n}{2r_n+1}-1\right)3r_n \geq r -3r_n-2r_n^2.
	\]
	So, for $r\geq R:= \lceil \frac{2r_n^2+3r_n}{1-\epsilon}\rceil$, we have $r-2r_n^2-3r_n\geq \epsilon r$, and thus
	\[
	\epsilon^2\bx(r)\geq r -2r_n^2-3r_n\geq \epsilon r.\qedhere
	\]
\end{proof}

\begin{prop}\label{p:prodspheres}
	For $R$ large enough, we have  $\prod_{r=3}^R(\sx(r)+1)>(\Delta-1)[\bx(4R+1)+1]^2$ for all $x\in X$.
\end{prop}

\begin{proof}
	In order to obtain lower bounds for the function $\prod_{r=3}^R(\sx(r)+1)$, let us consider the following optimization problem: given $\Delta,Q,R\in \N$  with  
	\begin{equation}\label{b}
	\Delta>2,\qquad R>3,\qquad Q>\Delta^2+R-1,
	\end{equation}
	minimize the function 
	\begin{equation}\label{f}
	f(a_1,\ldots,a_R) = \prod_{i=3}^R(a_i+1)
	\end{equation}
	for $a=(a_1,\ldots,a_R)\in (\mathbb{Z}^+)^R$ satisfying
	\begin{align}
	\label{a1leqdelta} a_1&\leq \Delta, \tag{C1}\\ \label{aileqai-1delta} a_i&\leq a_{i-1}(\Delta-1),\tag{C2}\\  \label{sumai}\sum_{i=1}^R a_i &= Q-1\tag{C3}
	\end{align}
	for $i=1,\ldots,R$.
	
	\begin{claim}\label{c:minimum}
		The above problem has a minimum $(a_1,\ldots,a_R)$ satisfying:
		\begin{enumerate}[(i)]
			\item \label{i:L}$a_1=\Delta$, and $a_2=\Delta(\Delta-1)$.
			\item \label{i:I}There is  $0\leq I\leq R-2$ such that the sequence $a_{2},\ldots, a_{2+I}$ is  increasing and  $a_i<\Delta(\Delta-1)$ for $i> 2+I$.
			\item \label{i:LplusI}For $3\leq i\leq 2+I$, we have $a_{i}+1>(a_{i-1}-1)(\Delta-1)$.
		\end{enumerate}
	\end{claim}
	Suppose that $(a_1,\ldots,a_R)$ is a minimum that does not satisfy~(\ref{i:L}),  let $n\in\{1,2\}$ be the first index such that $a_n<\Delta(\Delta-1)^{n-1}$, and let $m\geq 3$ be such that $a_m=\max\{a_i\mid i\geq 3\}$. Conditions~\eqref{a1leqdelta} and~\eqref{aileqai-1delta} yield
	\begin{equation}\label{a1a2}
	a_1+a_2\leq \Delta + \Delta(\Delta-1) =\Delta^2.
	\end{equation}
	If $a_i=1$ for all $i\geq 3$, then
	\[
	\sum_{i=1}^Ra_i=a_1+a_2+\sum_{i=3}^Ra_i\leq\Delta^2+R-2< Q-1
	\]
	by~\eqref{b}, contradicting~\eqref{sumai}; this shows that $a_m>1$.
	The sequence $(a'_1,\ldots,a'_R)$ given by 
	\[
	a'_i=\begin{cases}
	a_i+1 &\text{for}\ i=n,\\
	a_i-1 &\text{for}\ i=m,\\
	a_i &\text{otherwise.}
	\end{cases}
	\]
	still satifies~\eqref{a1leqdelta}--\eqref{sumai}, and clearly  $f(a'_1,\ldots,a'_R)<f(a_1,\ldots,a_R)$ since the index $n$ does not appear in~\eqref{f}. It follows that every minimum has to satisfy~(\ref{i:L}).
	
	Let us prove that we can obtain a minimum satisfying both~(\ref{i:L}) and~(\ref{i:I}). Let $(a_1,\ldots,a_R)$ be a minimum, and let $s$ be a permutation of $\{1,\ldots, R\}$ so that $s(1)=1$, $s(2)=2$, and 
	\[
	(a'_1,\ldots,a'_R)=(a_{s(1)},\ldots,a_{s(R)})
	\]
	satisfies~(\ref{i:I}); it is obvious that such a permutation always exists. Since $s$ leaves the subset $\{3,\ldots,R\}$ invariant and the function $f$ is symmetric in those indices, $(a'_1,\ldots,a'_R)$ is also a minimum if it  satisfies~\eqref{a1leqdelta}--\eqref{sumai}. 
	
	Let us prove that $(a'_1,\ldots,a'_R)$ satisfies~\eqref{a1leqdelta}--\eqref{sumai}: Condition~\eqref{a1leqdelta} holds because $s(1)=1$. In order to prove~\eqref{aileqai-1delta}, we begin by showing the following claim.
	\begin{claim}\label{c:ptwo}
		For every $i\in\{3,\ldots,R\}$ with $a_i> a_2$, there is some $j\in \{2,\ldots R\}$ such that $j\neq i$ and $a_2\leq a_{j}< a_i\leq (\Delta-1)a_{j}$.
	\end{claim}
	Let $l$ be an integer to be determined later, we are going to define a sequence of indices $m_1,\ldots,m_l$ in $\{2,\ldots,R\}$. Let
	\[
	m_1=\inf\{\,i\in \{2,\ldots,R\}\mid a_i\geq a_j\ \text{for all}\ 2\leq j\leq R\,\},
	\]
	and assume $a_{m_1}>a_2$, since otherwise the claim is vacuously true.
	Suppose now that, for $i>1$, we have defined $m_j$ for $1\leq j<i$. If $a_{m_{i-1}}=a_2$, then let $l=i-1$, so that $m_{i-1}$ is the last element in the sequence. If $a_{m_{i-1}}>a_2$, then let
	\[
	m_i=\inf\{\,i\in\{2,\ldots,m_{i-1}\}\mid a_i\geq a_j\ \text{for all}\ 2\leq j\leq m_{i-1}\,\}.
	\]
	It follows easily from the definition that $a_{m_i-1}<a_{m_i}$ for all $1\leq i<l$, and thus~\eqref{aileqai-1delta} yields
	\[
	(\Delta-1)^{-1}a_{m_i}\leq a_{m_i-1}<a_{m_i}.
	\]
	Now the claim follows from the trivial observation that, for every $i\in \{3,\ldots,R\}$ such that $a_2<a_i$, there is some $j\in\{1,\ldots,l-1\}$ such that $a_{m_{j+1}}\leq a_i\leq a_{m_j}$.
	 
	We resume the proof of~\eqref{aileqai-1delta}, so let $I$ be the largest non-negative integer so that $a'_2,\ldots a'_{2+I}$ is increasing. Recall that $a'_2= a_2$, and let $3\leq i\leq 2+I$. If $a'_i=a'_2$, then $a'_{i-1}=a'_2=a'_i$, so~\eqref{aileqai-1delta} is satisfied. If $a'_i>a'_2$, then by Claim~\ref{c:ptwo} there is some $j\in \{2,\ldots,R\}$ such that $a_2\leq a_j<a_{s(i)}\leq(\Delta-1)a_j$. Since $a_j>a_2$, we have $2\leq s^{-1}(j)\leq 2+I$ by~(\ref{i:I}). Also, the sequence $a'_2,\ldots,a'_{2+I}$ is increasing, so $a_j\leq a'_{i-1}$ and therefore $a'_{i}\leq(\Delta-1)a'_{i-1}$.
	Thus Condition~\eqref{sumai} is satisfied because the sum $\sum_{i=1}^Ra_i$ is invariant by permutations, and  we have obtained a minimum $(a'_1,\ldots,a'_R)$ that satisfies~(\ref{i:L}) and~(\ref{i:I}).
	
	Finally, suppose that $(a_1,\ldots,a_R)$ is a minimum satisfying~(\ref{i:L}) and~(\ref{i:I}), but not~(\ref{i:LplusI}). Let $n$ be an index such that $3\leq n\leq R-1$ and $a_{n}+1\leq (a_{n-1}-1)(\Delta-1)$, then one can easily  check that the solution $(a'_1,\ldots,a'_R)$ given by
	\[
	a'_i=\begin{cases}
	a_i-1 &\text{for}\ i=n-1,\\
	a_i+1 &\text{for}\ i=n,\\
	a_i &\text{otherwise.}
	\end{cases}
	\]
	still satifies~\eqref{a1leqdelta}--\eqref{sumai}. Furthermore, $a_{n+1}\geq a_n$ implies $(a_{n+1}+1)(a_n-1)<a_{n+1}a_n$, so $f(a'_1,\ldots,a'_R)<f(a_1,\ldots,a_R)$, contradicting the assumption that $(a_1,\ldots,a_R)$ was a minimum. This completes the proof of Claim~\ref{c:minimum}. 
	
	Our use of Claim~\ref{c:minimum} is that, for the purpose of obtaining lower bounds for the function $\prod_{r=3}^R(\sx(r)+1)$, we will assume that the sequence $\sx(r)$ satisfies Claim~\ref{c:minimum}(\ref{i:L})--(\ref{i:LplusI}). Thus~\eqref{a1a2} and Claim~\ref{c:minimum}(\ref{i:I}) yield 
	\begin{align}
	\sum_{r=3}^{2+I}\sx(r)&= \sum_{r=1}^{R}\sx(r)-\sum_{r=3+I}^{R}\sx(r)-\sum_{r=1}^{2}\sx(r)\notag\\
	&\geq \bx(R)-(R-2-I)\Delta(\Delta-1) - \Delta^{2}\notag\\
	&\geq \bx(R)-R\Delta(\Delta-1) - (\Delta-1)^2. \label{bxr-r}
	\end{align}
	
	By~\eqref{sxrplus1}, we have $\sx(2+r)\leq \sx(2)(\Delta-1)^{r}$ for $r=1,\ldots,I$, so 
	\begin{equation}\label{sx8delta-1}
	\sum_{r=1}^{I} \sx(2)(\Delta-1)^r =\sx(2)(\Delta-1)\frac{(\Delta-1)^{I}-1}{\Delta-2}\geq\sx(2)(\Delta-1)\frac{(\Delta-1)^{I}}{\Delta-1}=\Delta^2(\Delta-1)^{I} \geq\sum_{r=3}^{2+I} \sx(r).
	\end{equation}
	By Lemma~\ref{l:uniformsubexp} and~\eqref{bxr-r}, we have
	\begin{equation}\label{sum87plusi}
	\sum_{r=3}^{2+I}\sx(r)\geq \bx(R)/2
	\end{equation}
	for $R$ large enough and  all $x\in X$, and now~\eqref{sx8delta-1} and~\eqref{sum87plusi} yield 
	\begin{equation}\label{ibeta}
	(\Delta-1)^{I}\geq\bx(R)/2\Delta^2.
	\end{equation}
	
	From Claim~\ref{c:minimum}(\ref{i:LplusI}) we obtain by induction the following inequality  for $r=1,\ldots,I$.
	\begin{align*}
	\sx(2+r)&\geq \sx(2)(\Delta-1)^{r} -1-2\sum_{i=1}^{r-1}(\Delta-1)^i\\
			&\geq \sx(2)(\Delta-1)^{r} -1-2(\Delta-1)\frac{(\Delta-1)^{r-1}-1}{\Delta-2}\\
			&\geq (\Delta-1)^{r}(\sx(2)-\frac{2}{\Delta-2})-1.
	\end{align*}
	Since $\sx(2)=\Delta(\Delta-1)>2/(\Delta-2)+1$, we have
	\[\sx(2+r)\geq (\Delta-1)^{r}.\]
	Letting $C=1/2\Delta^2$,~\eqref{ibeta} yields
	\begin{multline}\label{bxlogbx}
	\prod_{r=3}^R (\sx(r)+1)\geq 	\prod_{r=3}^{2+I} (\sx(r)+1)\geq  \prod_{r=1}^{I}(\Delta-1)^{r}= ((\Delta-1)^{I+1})^{I/2}\geq[C\bx(R)]^{(\log_{\Delta-1}C\bx(R))/2}.
	\end{multline}
	We will split the last step on the proof in two cases depending on the growth type of $X$.
	
	\underline{Case 1: $X$ has symmetric exponential growth.} By Lemma~\ref{l:uniformexp}, there are $k,l,m\in\N$ such that $k\bx(ln)\geq e^n$ for all $x\in X$ and $n\geq m$. So, if $R\geq lm$, then~\eqref{bxlogbx} yields
	\[
	\prod_{r=3}^R (\sx(r)+1)\geq (Ck^{-1}e^{\lfloor R/l\rfloor})^{(\lfloor R/l\rfloor+\log Ck^{-1})/2}.
	\]
	Since $(Ck^{-1}e^{\lfloor R/l\rfloor})^{(\lfloor R/l\rfloor+\log Ck^{-1})/2}$ grows faster than $\Delta^{8R+7}$, we can assume that $R$ is large enough so that 
	\[
	\prod_{r=3}^R (\sx(r)+1)>\Delta^{8R+7}
	\]
	for all $x\in X$. Noting that $(\Delta-1)^2>3$, equation~\eqref{bxr} yields 
	\[
	\prod_{r=3}^R(\sx(r)+1)> \Delta[(\Delta-1)^{4R+3}]^{2}\geq (\Delta-1)[\bx(4R+1)]^2.
	\]

	\underline{Case 2: $X$ has symmetric sub-exponential growth.}
	By Hypothesis~\ref{h:lowerbound} and~\eqref{bxlogbx}, there is an increasing sequence $(r_n)$ such that
	\[
	\prod_{r=3}^{r_n} (\sx(r)+1)\geq (C2^{\sqrt{r_n}/4})^{(\log_{\Delta-1}C)/2 +(\sqrt{r_n}\log_{\Delta-1}2)/8}
	\]
	for all $x\in X$ and $n\in\N$.
	Letting $C'=\log_2 C$, we obtain 
	\[
	\prod_{r=3}^{r_n} (\sx(r)+1)\geq 2^{(C'\sqrt{r_n}\log_{\Delta-1}C)/8 +(C'r_n\log_{\Delta-1}2)/32}.
	\]
	By Lemma~\ref{l:uniformsub}, we also have
	\[
	\bx(4r_n+1)\leq 2^{(C'r_n\log_{\Delta-1}2)/128}
	\]
	for all $x\in X$ and $n$ large enough, so
	\[
	[\bx(4r_n+1)]^2\leq 2^{(C'r_n\log_{\Delta-1}2)/64}
	\]
	and thus
	\[
	\frac{\prod_{r=3}^{r_n} (\sx(r)+1)}{[\bx(4r_n+1)]^2}\geq 2^{(C'\sqrt{r_n}\log_{\Delta-1}C)/8+(C'r_n\log_{\Delta-1}2)/64}.
	\]
	The right-hand side of this equation goes to infinity as $n$ does, so 
	\begin{equation}\label{objective}
	\prod_{r=3}^{r_n} (\sx(r)+1)>(\Delta-1)[\bx(4r_n+1)]^2
	\end{equation}
	for $n$ large enough and all $x\in X$. Setting $R=r_n$ for some $n$ satisfying~\eqref{objective} completes the proof.
\end{proof}

\begin{proof}[Proof of Proposition~\ref{p:main}]
	The definitions of $A$ and $B$ in~\eqref{defnab} yield
	\begin{equation}\label{ab}
	\prod_{r=3}^R(\sx(r)+1) = \left[\prod_{r\in A}(\sx(r)+1)\right]\left[\prod_{r\in B}(\sx(r)+1)\right].
	\end{equation}
	We have $r-1\in B$ for every $r\in A$, so
	\begin{equation}\label{prodbproda}
	\prod_{r\in A}(\sx(r)+1)\leq(\Delta-1)\prod_{r\in B}(\sx(r)+1)
	\end{equation}
	because $\sx(r)\leq (\Delta-1)\sx(r-1)$ by~\eqref{sxrplus1}.
	The combination of~\eqref{ab} and~\eqref{prodbproda} then yields
	\[
	\prod_{r\in B}(\sx(r)+1)\geq \sqrt{\frac{\prod_{r=3}^R(\sx(r)+1)}{\Delta-1}},
	\]
	and the result follows from Proposition~\ref{p:prodspheres}.
\end{proof}

\section{Examples}\label{s:examples}

\subsection{A connected, locally finite graph with no coarsely distinguishing $2$-coloring}\label{ss:counter}

For $n\in\mathbb{Z}^+$, let $I_n=\{v_0,\ldots,v_n\}$ be a graph with edges $\{v_m,v_{m+1}\}$ for $m=0,\ldots,n-1$, and let $X=\{u_m\}_{m=1}^\infty$ be a graph with edges $\{u_m,u_{m+1}\}$ for $m\in\mathbb{Z}^+$. For every $n\in\mathbb{Z}^+$, take $2^n+1$ copies of $I_n$ and denote them by 
\[
I_n^i=\{ \, v_m^i\mid i=0,\ldots,n \, \}, \qquad i=1,\ldots,2^n+1.
\]
For every $n$ and $i$, glue the graph $I_n^i$ to $X$ by identifying the points $u_n$ and $v_0^i$; denote the resulting graph by $Y$ (see Figure~\ref{f:counter}), and let $Y_n$ be the full subgraph whose vertex set is the image of $\bigcup_i I_n^i$ by the quotient map. 

\begin{figure}[th]
	\includegraphics[width=0.5\textwidth]{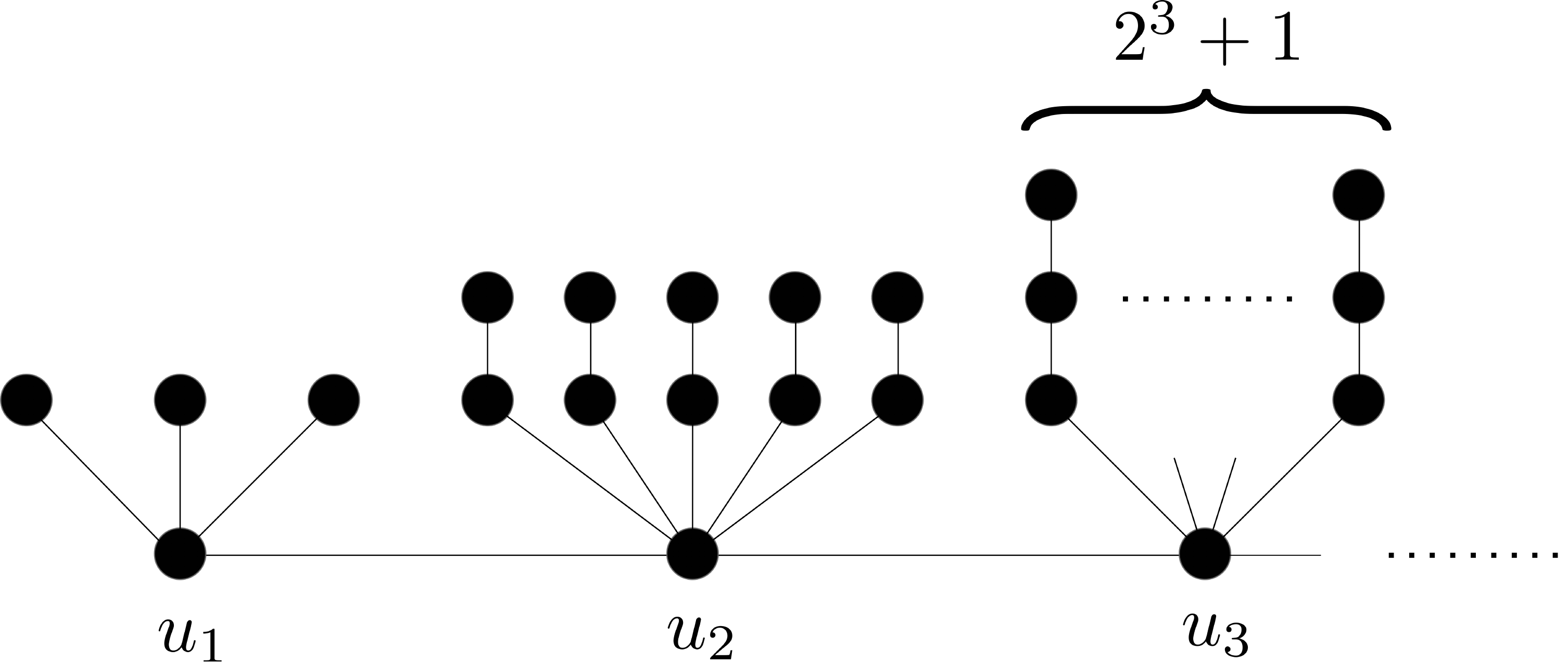}
	\caption{A graph without coarsely distinguishing $2$-colorings}
	
	\label{f:counter}
\end{figure}

Let $\phi$ be an arbitrary $2$-coloring of $Y$. Since we have $2^n+1$ copies of $I_n$ glued to $u_n$ $(n\in\mathbb{Z}^+)$, by the pigeonhole principle there are at least two indices $i(n)\neq j(n)$ such that the restrictions of $\phi$ to $I_n^{i(n)}$ and $I_n^{j(n)}$ are equal. So there exists an isomorphism $f_n$ of $Y_n$ that preserves $\phi$ and maps  $I_n^{i(n)}$ to $I_n^{j(n)}$, and therefore $d(f(v_n^{i(n)}),v_n^{i(n)})=2n$. Choose such an isomorphism $f_n$ for every $n\in \mathbb{Z}^+$, and combine them into an isomorphism $f$ of $Y$ preserving $\phi$. Since $d(f(v_n^{i(n)}),v_n^{i(n)})=2n$ for all $n\in\mathbb{Z}^+$, the map $f$ is not close to the identity. Note that the vertex $u_n$ has degree $4+2^n$, so $\deg Y=\infty$ and hence $Y$ does not have symmetric growth.

\subsection{Graphs with infinite motion but finite geometric motion}\label{ss:infmotion}

Perhaps the simplest example of a connected locally finite graph $X$ with $m(X)=\infty$ and $\gm(X)<\infty$ is shown in Figure~\ref{f:motion}. This graph has symmetric linear growth. The only non-trivial automorphism $f$ is the obvious one interchanging the horizontal rays starting at $y$ and $z$,  and it is easy to check that $d(x,f(x))\leq 1$ for all $x\in X$.

\begin{figure}[tbh]
    \includegraphics{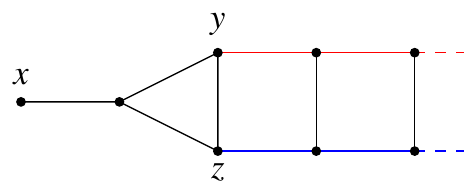}
	\caption{Example of a graph $X$ with $m(X)=\infty$ and $\gm(X)<\infty$}
	\label{f:motion}
\end{figure}

We can modify this example to obtain graphs with infinite motion, finite geometric motion, and larger symmetric growth. For example, let $T_3$ be the regular tree of degree $4$, and let $\phi\colon T_3\to \{0,1\}$ be an distinguishing coloring. Substitute each edge in $T_3$ by a ``gadget" depending on the colors of the incident vertices (see Figure~\ref{f:substitution}). In this way we obtain a graph $Y$ with $\Aut(Y)=\{\id_Y\}$ and symmetric exponential growth. Moreover, we can identify $T_3$ with the subset $\overline Y$ of $Y$ consisting of vertices of degree $4$. Gluing one copy of $X$ to each vertex $y\in \overline Y$ by identifying it with $x$, we obtain a graph with infinite motion, finite geometric motion, and symmetric exponential growth.

\begin{figure}[tbh]
	\includegraphics{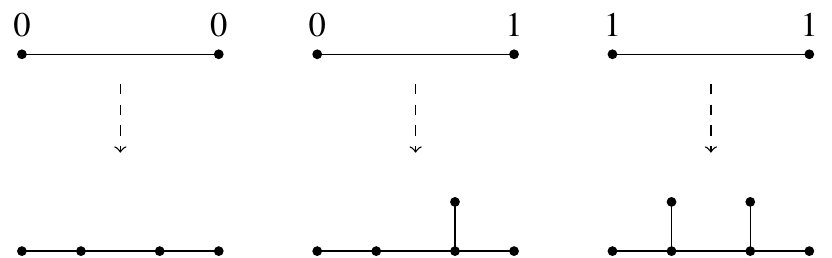}
	\caption{Substituting each edge in $T_4$ by a graph}
	\label{f:substitution}
\end{figure}

\subsection{Diestel-Leader Graphs}\label{ss:dl} The Diestel-Leader graphs $\DL(p_1,\ldots,p_n)$ are defined for $n,p_1,\ldots,p_n\geq 2$. For the sake of simplicity, however, we will restrict our attention to the case $n=2$; at any rate, the following discussion can be easily adapted to include the case $n>2$. In order to define $\DL(p,q)$, let $T_p$ and $T_q$ be the regular trees of degree $p+1$ and $q+1$, respectively. For $i=p,q$, choose a root $o_i\in T_i$ and fix an end $\omega_i$ of $T_i$. These choices induce \emph{height} or \emph{Busemann functions} $\h_i\colon T_i\to \mathbb{Z}$, and then
\[
\DL(p,q):=\{\,(x,y)\in T_p\times T_q\mid \h_p(x)+\h_q(y)=0\,\}.
\]
Let us  $(x,y)\in\DL(p,q)$ as $xy$ for the sake of clarity, and let  $xE_i y$ denote that $x$ and $y$ are adjacent in $T_i$, then the graph structure $E$ in $\DL(p,q)$ is defined by
\[
xyE x'y' \quad \text{if and only if} \quad xE_p x'\; \text{and}\; yE_q y'. 
\]
This  yields
\begin{equation}\label{distancedl}
d_{\DL(p,q)}(xy,x'y')\geq \max\{d_{T_p}(x,x'),d_{T_q}(y,y')\}\geq\max\{|\h(x)-\h(x')|,|\h(y)-\h(y')|\}.
\end{equation}

For $i=p,q$, let $\Aff(T_i)$ be the subgroup of automorphisms of $T_i$ that fix $\omega_i$. For every $f\in\Aff(T_i)$, the quantity $\h(f(x))-\h(x)$ is independent of $x\in T_i$, and we will denote it by $\h(f)$.  
Let 
\[
\mathcal{A}_{p,q}=\{\, (f,f')\in \Aff(T_p)\times\Aff(T_q) \mid \h_p(f)+\h_q(f')=0\,\}.
\]
\begin{lem}[{\cite[Thm.2.7.]{Bartholdi}, \cite[Prop.~3.3]{Bertacchi}}]\label{l:aut}
	If $p\neq q$, then $\Aut(\DL(p,q))\cong\mathcal{A}_{p,q}$. For $p=q$, the group $\Aut(\DL(p,p))$ is generated by $\mathcal{A}_{p,p}$ and the map $\sigma\colon xy\mapsto yx$.
\end{lem}

Let us prove that $\DL(p,q)$ satisfies the hypothesis of Theorem~\ref{t:motion}.

\begin{lem}
	The group $\Aut(\DL(p,q))$ has infinite motion, and the stabilizer $S_{o_po_q}$ has infinite geometric motion. 
\end{lem} 
\begin{proof}
Let $a=(f,f')\in \mathcal A_{p,q}$. If $a\neq\id$, then at least one of $f$, $f'$ is non-trivial, say $f$. Therefore $f$ is a non-trivial automorphism of a regular tree, hence $m(f)=m(a)=\infty$. If moreover $a\in S_{o_po_q}$, then $f(o_p)=o_p$, and therefore $\gm(f)=\infty$ when considered as an automorphism of $T_p$ (it is elementary to check that stabilizers in regular tres have infinite geometric motion). Now~\eqref{distancedl} yields $\gm(a)=\infty$, proving the result when $p\neq q$ by Lemma~\ref{l:aut}.

If $p=q$, then every automorphism which is not in $\mathcal A_{p,q}$ can be written as $\sigma a$, where $a=(f,f')\in\mathcal{A}_{p,p}$ and $\sigma$ is the map $xy\mapsto yx$. Since $f(o_p)=f'(o_p)=o_p$, we have $\h(f)=\h(f')=0$. Let $x_ny_n$ be a sequence in $\DL(p,p)$ with $\h_p(x_n)=-\h_p(y_n)=n$. Then 
\begin{multline*}
d(x_ny_n,\sigma a(x_ny_n)) = d(x_ny_n,f'(y_n)f(x_n))\geq |\h_p(x_n)- \h_p(f'(y_n))|=|\h_p(x_n)- \h_p(y_n)-\h_p(f)|\geq 2n-\h_p(f),
\end{multline*}
so $\gm(a)=m(a)=\infty$.
\end{proof}

\subsection{Graphs with bounded cycle length}\label{ss:boundedcycle} A \emph{cycle} of length $n\in\N$ in a graph is a path $\sigma$ of length $n$ with $\sigma(0)=\sigma(n)$ and $\sigma(i)\neq\sigma(j)$ for $0\leq i<j<n$. A graph $X$ has \emph{bounded cycle length} if there is $L\in\N$ such that every cycle in $X$ has length $\leq L$. It is not difficult to prove that all graphs of bounded cycle length are hyperbolic in the sense of Gromov. There are in the literature several non-equivalent definitions of the \emph{free product} of graphs, see e.g.~\cite{CarterTornierWillis}; one can easily check, however, that the following result holds for any of the definitions: The free product of a finite family of graphs of bounded cycle length has bounded cycle length. In particular, the free product of a finite family of finite graphs has bounded cycle length.

\begin{lem}[Cf.~{\cite[Lem.~3.6]{Lehner}}]\label{l:disjointray}
Let $X$ be a connected locally finite graph with infinite motion,  let $x\in X$, and let $f\in S_x$. Then there is a ray $\gamma\colon\N\to X$ such that $\gamma(0)=f(\gamma(0))$ and $\im(\gamma)\cap\im(f\circ\gamma)=\{\gamma(0)\}$.
\end{lem}
\begin{proof}
See the proof of~\cite[Lem.~3.6]{Lehner}.
\end{proof}

\begin{prop}
If $X$ has infinite motion and bounded cycle length, then every vertex stabilizer has infinite geometric motion.
\end{prop}
\begin{proof}
Let $x\in X$ and let $f\in S_x$. By Lemma~\ref{l:disjointray}, there is a ray $\gamma$ such that, if we let $\gamma'=f(\gamma)$, then $\gamma(0)=\gamma'(0)$ and $\im(\gamma)\cap \im(\gamma')=\{\gamma(0)\}$. 
For $n\in\mathbb{Z}^+$, choose geodesic paths $\sigma_{n}$ from $\gamma(n)$ to $\gamma'(n)$. Let $m_n$ be the largest integer such that $\sigma_n(m_n)\in\im\gamma$, and let $m'_n$ be the least integer such that $\sigma_n(m'_n)\in\im\gamma'$; clearly $m_n,m'_n\leq d(\gamma(n),\gamma'(n))$.
The  triangle $Z_n$ with sides 
\[
(\gamma(0),\ldots, \gamma(i)=\sigma(m_n)),\quad  (\sigma(m_n), \sigma(m_n+1),\ldots,\sigma(m'_n)),\quad \text{and} \quad (\gamma'(j)=\sigma(m'_n),\gamma'(j-1),\ldots, \gamma'(0))
\]
determines a cycle of length  $\geq 2n-2d(\gamma(n),\gamma'(n))$. Now the assumption that $X$ has bounded cycle length yields $\lim d(\gamma(n),\gamma'(n))=d(\gamma(n),f(\gamma(n))=\infty$, and the result follows.
\end{proof}

\subsection*{Acknowledgements}  Part of this work was carried out during the tenure of a Canon Foundation in Europe Research Fellowship by B.L. H.N. is partly supported by JSPS KAKENHI Grant Number 17K14195. The authors are supported by the Program for the Promotion of International Research by Ritsumeikan University.

\bibliographystyle{ieeetr}
\bibliography{infmotion}

\end{document}